\documentclass[oneside, article, a4paper, oldfontcommands]{memoir}
\usepackage[T1]{fontenc}
\usepackage[utf8]{inputenc}
\usepackage[left = 3.5cm, right=3.5cm, top=3cm, bottom=3.5cm]{geometry}

\usepackage{amsmath}
\usepackage{amsthm}
\usepackage{amssymb}
\usepackage[dvipsnames]{xcolor}
\usepackage{newpxtext}
\usepackage[fracspacing]{newpxmath}
\usepackage{eucal}
\usepackage{microtype}
\usepackage[labelfont=bf]{caption}
\usepackage{subcaption}
\usepackage{enumitem}
\usepackage{parskip}
\usepackage{tikz}
\usepackage{tikz-cd}
\usepackage{url}
\usepackage[colorlinks=true,linkcolor=BrickRed,anchorcolor=BrickRed,citecolor=BrickRed,filecolor=BrickRed,menucolor=BrickRed,runcolor=BrickRed,urlcolor=BrickRed]{hyperref}
\usepackage{tikzit}

\usepackage[small, sc, center, compact]{titlesec}
\titleformat{\paragraph}
{\bfseries}
{\thesection}{1em}{}
\setsecnumdepth{subsection}
\counterwithout{section}{chapter}

\newtheoremstyle{note}%
{}{}%
{\upshape}{}%
{\bfseries}{.}%
{ }%
{\thmnumber{#2}}

\newtheorem{thm}{Theorem}[section]
\newtheorem{prop}[thm]{Proposition}
\newtheorem{lem}[thm]{Lemma}

\theoremstyle{definition}
\newtheorem{definition}[thm]{Definition}
\newtheorem{construction}[thm]{Construction}
\newtheorem{example}[thm]{Example}
\newtheorem{remark}[thm]{Remark}

\newtheorem{para-todo}[thm]{Todo}
\theoremstyle{note}

\usetikzlibrary{backgrounds, arrows, arrows.meta}

\nobibintoc
\newcommand{\R}{\mathbb{R}}
\newcommand{\ord}[1]{\lbrack #1 \rbrack}
\newcommand{\op}{\textup{op}}
\newcommand{\cat}[1]{\mathcal{#1}}

\newcommand{\Truss}{\varmathbb{T}}
\newcommand{\ETruss}{\textup{ET}}

\newcommand{\Mesh}{\varmathbb{M}}

\newcommand{\BMesh}{\textup{BM}}
\newcommand{\EMesh}{\textup{EM}}

\newcommand{\IntOpen}{(-1, 1)}

\newcommand{\GrContr}{\textup{Gr}^-}

\newcommand{\FinOrd}{\Delta}
\newcommand{\StrictInt}{\nabla}

\newcommand{\Cat}{\mathcal{C}\textup{at}}
\newcommand{\Fun}{\textup{Fun}}
\newcommand{\Hom}{\textup{Hom}}
\newcommand{\Map}{\textup{Map}}

\newcommand{\Strat}{\mathcal{S}\textup{trat}}
\newcommand{\StratPre}{\textup{Strat}}
\newcommand{\sSet}{\textup{sSet}}
\newcommand{\Pos}{\textup{Pos}}
\renewcommand{\Top}{\textup{Top}}
\newcommand{\DeltaStrat}[1]{\textup{st}\ord{#1}}

\newcommand{\DeltaTop}[1]{{|\Delta\ord{#1}|}}
\newcommand{\HornTop}[2]{{|\Lambda_{#2}\ord{#1}|}}

\newcommand{\CatInfty}{\Cat_\infty}
\newcommand{\Space}{\mathcal{S}\textup{pace}}

\newcommand{\PSh}{\mathcal{P}}

\newcommand{\Sing}{\textup{Sing}}
\newcommand{\SingStrat}{\mathbb{S}\textup{ing}}

\newcommand{\sing}{\textup{sing}}
\newcommand{\reg}{\textup{reg}}

\newcommand{\id}{\textup{id}}

\newcommand{\Exit}{\textup{Exit}}

\newcommand{\pullback}{\ar["\lrcorner", phantom, very near start, dr]}

\tikzstyle{singularity}=[fill=black, draw=black, shape=circle, minimum size=3pt, inner sep=0pt, outer sep=0pt]
\tikzstyle{dual-singularity}=[fill=gray, draw=gray, shape=circle, minimum size=3pt, inner sep=0pt, outer sep=0pt]
\tikzstyle{truss-dim}=[fill=none, draw=none, shape=circle, inner sep=0pt, outer sep=0.6pt, font={\tiny}]
\tikzstyle{generator-label}=[fill=white, draw=black, shape=circle, inner sep=1pt, outer sep=0pt, font={\tiny}]
\tikzstyle{braid-mask}=[fill=white, draw=none, shape=circle]

\tikzstyle{dual-edge}=[-, draw=gray, line width=0.65pt]
\tikzstyle{edge}=[-, line width=0.65pt]
\tikzstyle{horizontal-edge}=[-, draw=black, line width=0.65pt, dotted]
\tikzstyle{edge-arrow}=[->, line width=0.5pt]

\title{Framed Combinatorial Topology with Labels in $\infty$-Categories}
\author{Lukas Heidemann}

\tikzset{
  wire/.style={solid, thick, line cap = round},
  wire-mask/.style={solid, line width = 2mm, white},
  vertex/.style={circle, black},
  boundary/.style={dashed, gray},
  background/.style={gray!20}
}

\begin{document}
\maketitle

\begin{abstract}
  Framed combinatorial topology is a recent approach to tame geometry which expresses higher-dimensional stratified spaces via tractable combinatorial data.
  The resulting theory of spaces is well-behaved and computable.
  In this paper we extend FCT by allowing labelling systems of meshes and trusses in $\infty$-categories, and build an alternative model of FCT by constructing $\infty$-categories that classify meshes and trusses.
  This will serve as the foundation for future work on models of higher categories based on generalised string diagrams
  and the study of generalised tangles.
\end{abstract}

\tableofcontents*

\section{Introduction}

Framed combinatorial topology~\cite{fct} is a recent approach to tame geometry.
Well-behaved stratified spaces are equipped with a canonical ordering of coordinate dimensions, called a framing,
leading to a very tractable theory.
Meshes represent a stratified space as a sequence of $1$-dimensional bundles, each adding one additional coordinate in the order of the framing;
trusses are purely combinatorial representations of framed isomorphism classes of meshes which faithfully capture all the geometric structure of meshes in a form that is amenable to computer implementation.
In this paper we extend FCT with labels in $\infty$-categories.

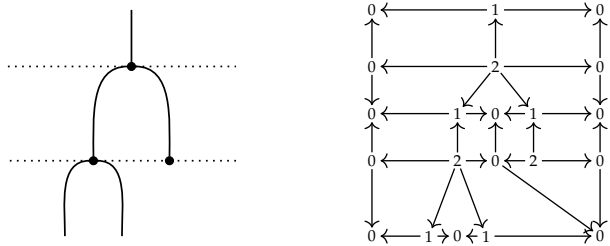
\begin{figure}[h]
  \centering
  \begin{tikzpicture}
	\begin{pgfonlayer}{nodelayer}
		\node [style=singularity] (2) at (1.25, -1) {};
		\node [style=none] (3) at (0, -3) {};
		\node [style=none] (4) at (-1.5, -3) {};
		\node [style=singularity] (5) at (-0.75, -1) {};
		\node [style=none] (6) at (0.25, 3) {};
		\node [style=singularity] (7) at (0.25, 1.5) {};
		\node [style=none] (8) at (-3, -1) {};
		\node [style=none] (9) at (3, -1) {};
		\node [style=none] (10) at (-3, 1.5) {};
		\node [style=none] (11) at (3, 1.5) {};
	\end{pgfonlayer}
	\begin{pgfonlayer}{edgelayer}
		\draw [style=horizontal-edge] (5) to (8.center);
		\draw [style=horizontal-edge] (5) to (2);
		\draw [style=horizontal-edge] (9.center) to (2);
		\draw [style=horizontal-edge] (7) to (11.center);
		\draw [style=horizontal-edge] (10.center) to (7);
		\draw [style=edge, in=-180, out=90] (4.center) to (5);
		\draw [style=edge, in=0, out=90] (3.center) to (5);
		\draw [style=edge, in=-180, out=90] (5) to (7);
		\draw [style=edge] (7) to (6.center);
		\draw [style=edge, in=90, out=0] (7) to (2);
	\end{pgfonlayer}
\end{tikzpicture}
  \qquad\qquad
  \begin{tikzpicture}
	\begin{pgfonlayer}{nodelayer}
		\node [style=truss-dim] (2) at (1.25, -1) {2};
		\node [style=truss-dim] (3) at (0, -3) {1};
		\node [style=truss-dim] (4) at (-1.5, -3) {1};
		\node [style=truss-dim] (5) at (-0.75, -1) {2};
		\node [style=truss-dim] (6) at (0.25, 3) {1};
		\node [style=truss-dim] (7) at (0.25, 1.5) {2};
		\node [style=truss-dim] (8) at (-3, -1) {0};
		\node [style=truss-dim] (9) at (3, -1) {0};
		\node [style=truss-dim] (10) at (-3, 1.5) {0};
		\node [style=truss-dim] (11) at (3, 1.5) {0};
		\node [style=truss-dim] (12) at (-3, -3) {0};
		\node [style=truss-dim] (13) at (-3, 0.25) {0};
		\node [style=truss-dim] (14) at (-3, 3) {0};
		\node [style=truss-dim] (15) at (3, 3) {0};
		\node [style=truss-dim] (16) at (3, 0.25) {0};
		\node [style=truss-dim] (17) at (3, -3) {0};
		\node [style=truss-dim] (18) at (-0.75, -3) {0};
		\node [style=truss-dim] (20) at (0.25, -1) {0};
		\node [style=truss-dim] (21) at (-0.75, 0.25) {1};
		\node [style=truss-dim] (22) at (1.25, 0.25) {1};
		\node [style=truss-dim] (23) at (0.25, 0.25) {0};
	\end{pgfonlayer}
	\begin{pgfonlayer}{edgelayer}
		\draw [style=edge-arrow] (7) to (22);
		\draw [style=edge-arrow] (7) to (21);
		\draw [style=edge-arrow] (5) to (21);
		\draw [style=edge-arrow] (21) to (23);
		\draw [style=edge-arrow] (22) to (23);
		\draw [style=edge-arrow] (22) to (16);
		\draw [style=edge-arrow] (21) to (13);
		\draw [style=edge-arrow] (7) to (10);
		\draw [style=edge-arrow] (6) to (14);
		\draw [style=edge-arrow] (6) to (15);
		\draw [style=edge-arrow] (5) to (8);
		\draw [style=edge-arrow] (4) to (12);
		\draw [style=edge-arrow] (4) to (18);
		\draw [style=edge-arrow] (3) to (18);
		\draw [style=edge-arrow] (3) to (17);
		\draw [style=edge-arrow] (9) to (17);
		\draw [style=edge-arrow] (9) to (16);
		\draw [style=edge-arrow] (11) to (16);
		\draw [style=edge-arrow] (11) to (15);
		\draw [style=edge-arrow] (10) to (14);
		\draw [style=edge-arrow] (10) to (13);
		\draw [style=edge-arrow] (8) to (13);
		\draw [style=edge-arrow] (8) to (12);
		\draw [style=edge-arrow] (5) to (20);
		\draw [style=edge-arrow] (2) to (20);
		\draw [style=edge-arrow] (2) to (9);
		\draw [style=edge-arrow] (5) to (4);
		\draw [style=edge-arrow] (5) to (3);
		\draw [style=edge-arrow] (2) to (22);
		\draw [style=edge-arrow] (7) to (11);
		\draw [style=edge-arrow] (7) to (6);
		\draw [style=edge-arrow] (20) to (23);
		\draw [style=edge-arrow] (20) to (17);
	\end{pgfonlayer}
\end{tikzpicture}
  \caption{A string diagram expressed both geometrically as a mesh and combinatorially as a truss.
  The two representations are equivalent up to a contractible choice of diagram layout.\label{fig:diagram-mesh-truss}}
\end{figure}

\textbf{Motivation.}
A central application of FCT is in models of higher-dimensional categories
based on a generalisation of string diagrams to arbitrary dimensions~\cite{Dorn_2022, zigzag-contraction, Heidemann_2022}.
Meshes detect critical points of braidings and higher-dimensional changes in relative positioning (see Figure~\ref{fig:mesh-detect-braid}),
making them useful in the formulation of a semi-strict theory of $(\infty, n)$-categories with weak interchange laws arising from geometry.
FCT helps translate between geometric representations of diagrams and a combinatorial representation based on trusses (see Figure~\ref{fig:diagram-mesh-truss}), which has been implemented in the proof assistant \url{homotopy.io}.

\begin{figure}[h]
  \centering
  \begin{tikzpicture}
	\begin{pgfonlayer}{nodelayer}
		\node [style=none] (0) at (-1, 3) {};
		\node [style=none] (1) at (-1, -3) {};
		\node [style=none] (2) at (1, -3) {};
		\node [style=none] (3) at (1, 3) {};
		\node [style=none] (4) at (0, 0) {};
		\node [style=none] (5) at (-3, 0) {};
		\node [style=none] (6) at (3, 0) {};
	\end{pgfonlayer}
	\begin{pgfonlayer}{edgelayer}
		\draw [style=edge, bend left=15] (3.center) to (4.center);
		\draw [style=edge, bend right=15] (4.center) to (1.center);
		\node [style=braid-mask] (7) at (0, 0) {};
		\draw [style=horizontal-edge] (5.center) to (4.center);
		\draw [style=horizontal-edge] (4.center) to (6.center);
		\draw [style=edge, bend right=15] (0.center) to (4.center);
		\draw [style=edge, bend left=15] (4.center) to (2.center);
	\end{pgfonlayer}
\end{tikzpicture}
  \caption{Framed combinatorial topology detects the critical point in the braid in which the two wires move across each other.
  This generalises to combinatorially capture the weak structure of higher dimensional categories, such as the Yang-Baxterator, the Zamolodchikov tetrahedron, etc.\label{fig:mesh-detect-braid}}
\end{figure}
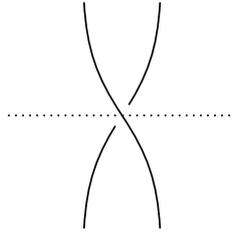

However, semi-strict approaches to higher category theory based on string diagrams have not yet been proven sound or complete beyond dimension three~\cite{barrett-graycategories}.
Via trusses and meshes labelled in $\infty$-categorical labels and a unified treatment of truss morphisms and bordisms, the formulation of FCT in this paper is designed to ease comparison with accepted homotopy theoretic models of higher categories.
In particular, we will explore in future work how an inclusion functor from the $n$-fold product $\Delta^{\times n}$ into $n$-trusses enables a comparison with $n$-fold Segal spaces~\cite{rezk2001model, lurie2008classification, lurie2009infinity, barwick2011unicity}.

FCT also has promising applications to geometry and singularity theory,
beginning with the study of manifold diagrams and tame tangles laid out in~\cite{dorn2022manifold}.
Through $\infty$-categorical labels, we can express local structures imposed on the normal and tangent bundles of manifolds realised as meshes.
Together with the interpretation of meshes as higher-dimensional string diagrams, we believe that this is a natural context to study generalised tangles.

Since this paper is a foundational paper, we will not discuss applications in detail.
For a more detailed discussion of applications, see~\cite{fct, Dorn_2022, Heidemann_2022} and future work of the author.

\textbf{Results.}
Trusses and meshes as defined in~\cite{fct} can be equipped with labels in a $1$-category.
In this paper we demonstrate how this can be generalised to labels in an $(\infty, 1)$-category,
which allows for labels arising from geometric data.
For every $\infty$-category $\cat{C}$ we characterise an $\infty$-category $\Truss(\cat{C})$ such that functors $\cat{B} \to \Truss(\cat{C})$ are equivalent to truss bundles on $\cat{B}$ labelled in $\cat{C}$.
Further we construct an $\infty$-category $\Mesh(\cat{C})$ such that for any (well-behaved)
stratified space $B$ there is an equivalence between functors $\Exit(B) \to \Mesh(\cat{C})$ and mesh bundles on $B$ labelled in $\cat{C}$.
We prove that there is a natural equivalence $\Truss(\cat{C}) \simeq \Mesh(\cat{C})$ and so the theories of trusses and meshes arising from our constructions are equivalent.
Since $\Truss(\cat{C})$ and $\Mesh(\cat{C})$ are $\infty$-categories they can serve as labels themselves, allowing us to obtain $n$-trusses and $n$-meshes $\Truss^n(\cat{C}) \simeq \Mesh^n(\cat{C})$.

\textbf{Related Work.}
The framed combinatorial topology project~\cite{fct} first defined trusses and meshes.
The results in this paper were originally developed as an evolution of the zigzag construction for the project set out in~\cite{zigzag-contraction, Heidemann_2022} which aims to study the theoretical and algorithmic underpinnings of the proof assistant \texttt{homotopy.io} and were arrived at independently.
To avoid a splintering of terminology we have decided to adopt the terms trusses and meshes from the latter, although some care must be taken since our definitions are not exactly equivalent.
We believe that the present paper still constitutes a valuable contribution over the discussion in~\cite{fct} due to its considerably different approach and the capacity to define labels in $\infty$-categories.

Our characterisation of the $\infty$-category of meshes as a sheaf on stratified spaces is 
inspired by the striation sheaves and constructible bundles from~\cite{ayala2018stratified}.
Since the notion of stratified space employed in this paper is more general, we have to rebuild some part of the theory.

\textbf{Acknowledgements.}
The author is grateful to
Christoph Dorn,
Christopher Douglas,
Eric Finster,
Nick Hu,
Chiara Sarti,
Calin Tataru,
Jamie Vicary and
Vincent Wang
for helpful discussions
and to Jamie Vicary in particular for his excellent support as a DPhil supervisor.

\textbf{Notation.}
We write $\ord{n}$ for the finite ordinal with elements $\{ 0, \ldots, n \}$.
$\Delta$ is the category of finite non-empty ordinals and order-preserving maps between them. 
The category of strict intervals $\StrictInt$ is the subcategory of $\Delta$ consisting of the ordinals $\ord{n}$ for $n \geq 1$ and order-preserving maps which preserve the minimum and maximum.
There is an isomorphism $\Delta^\op \cong \StrictInt$.

We use the term $\infty$-category for $(\infty, 1)$-categories in the sense of quasicategories or complete Segal spaces, but aim to remain mostly independent of any particular model.
We let $\CatInfty$ denote the (large) $\infty$-category of small $\infty$-categories and $\Space$ the (large) $\infty$-category of spaces.
For any $\infty$-category $\cat{C}$ and pair of objects $X, Y$ in $\cat{C}$ we write $\cat{C}(X, Y)$ for the space of maps from $X$ to $Y$.
When $\cat{C}, \cat{D}$ are $\infty$-categories, then $\Fun(\cat{C}, \cat{D})$ denotes the $\infty$-category of functors $\cat{C} \to \cat{D}$ whereas (in accordance with the notation for mapping spaces) the space $\CatInfty(\cat{C}, \cat{D})$ is the core of $\Fun(\cat{C}, \cat{D})$.
Given a functor $F : \cat{C}^\op \to \CatInfty$ we denote by $\GrContr(F) \to \cat{C}$ the contravariant Grothendieck construction or unstraightening.
We use $\downarrow$ to denote comma and slice $\infty$-categories.

\section{Trusses}

The fundamental ingredient in the definition of trusses are truss bundles,
which are functors of $\infty$-categories which arise as pullbacks of the functor $\ETruss \to \FinOrd$ defined as follows.

\begin{definition}\label{def:truss}
  The objects of the category $\ETruss$ are \textit{singular objects} of
  the form $s_i \ord{n}$ for $0 \leq i < n$ and \textit{regular objects} $r_i
    \ord{n}$ for $0 \leq i \leq n$.
  The morphisms of $\ETruss$ are defined as
  \begin{align}
    \label{eqn:truss-regular-regular}
    \ETruss(r_i \ord{n}, r_j \ord{m}) & = \{ \alpha : \ord{n} \to \ord{m} \mid \alpha(i) = j \}                       \\
    \label{eqn:truss-singular-singular}
    \ETruss(s_i \ord{n}, s_j \ord{m}) & = \{ \alpha : \ord{n} \to \ord{m} \mid \alpha(i) \leq j < \alpha(i + 1) \}    \\
    \label{eqn:truss-singular-regular}
    \ETruss(s_i \ord{n}, r_j \ord{m}) & = \{ \alpha : \ord{n} \to \ord{m} \mid \alpha(i) \leq j \leq \alpha(i + 1) \} \\
    \label{eqn:truss-regular-singular}
    \ETruss(r_i \ord{n}, s_j \ord{m}) & = \varnothing
  \end{align}
  and compose as the underlying order-preserving maps. There is a
  canonical forgetful functor $\ETruss \to \FinOrd$ called the
  \textit{universal truss bundle} which sends both
  $s_i \ord{n}$ and $r_i \ord{n}$ to $\ord{n}$.
\end{definition}

\begin{example}
  The fibre of $\ETruss \to \FinOrd$ over an ordinal $\ord{n}$ is a sequence of spans
  \[
    \begin{tikzcd}
      r_0\ord{n} &
      s_0\ord{n} \ar[l] \ar[r] &
      r_1\ord{n} &
      \cdots \ar[l] \ar[r] &
      r_{n - 1}\ord{n} &
      s_{n - 1}\ord{n} \ar[l] \ar[r] &
      r_n\ord{n}
    \end{tikzcd}
  \]
  alternating between regular and singular objects.
  In particular, the fibre over the terminal ordinal $\ord{0}$ consists of a single object $r_0 \ord{0}$.
\end{example}
\begin{example}
  The fibre of $\ETruss \to \FinOrd$ over an order-preserving map $\alpha : \ord{n} \to \ord{m}$ consists of two sequences of spans of lengths $n$ and $m$, respectively, connected together to form a planar diagram.
  In the following diagrams we omit the arrows from singular to regular objects which arise as composites of other arrows to avoid clutter.
  \begin{enumerate}
    \item
          The outer face map $\ord{1} \to
            \ord{2}$ defined by $(0 \mapsto 1, 1 \mapsto 2)$ embeds a span into a sequence of spans.
          We will use this to define embeddings of a subdiagram into a larger diagram.
          \[
            \begin{tikzcd}
              &&
              r_0\ord{1}
              \ar[d] &
              s_0\ord{1}
              \ar[d]
              \ar[l] \ar[r] &
              r_1\ord{1} \ar[d] \\
              r_0\ord{2} &
              s_0\ord{2} \ar[l] \ar[r] &
              r_1\ord{2} &
              s_1\ord{2} \ar[l] \ar[r] &
              r_2\ord{2}
            \end{tikzcd}
          \]
    \item
          The degeneracy map $\ord{1} \to
            \ord{0}$ merges two regular levels by removing the singular level between them.
          This will be used to model identity cells.
          \[
            \begin{tikzcd}
              r_0\ord{1} \ar[dr] &
              s_0\ord{1} \ar[l] \ar[r] &
              r_1\ord{1} \ar[dl] \\
              &
              r_0\ord{0}
            \end{tikzcd}
          \]
    \item
      The inner face map $\ord{1} \to \ord{2}$ defined by $(0 \mapsto 0, 1 \mapsto 2)$ splits a singular object into two (due to the way maps between singular objects are defined in Equation~\ref{eqn:truss-singular-singular}) and inserts a regular level between them.
          \[
            \begin{tikzcd}
              r_0\ord{1} \ar[d] &&
              s_0\ord{1}
              \ar[dl]
              \ar[dr]
              \ar[ll]
              \ar[rr] 
              &&
              r_1\ord{1}
              \ar[d]
              \\
              r_0\ord{2} &
              s_0\ord{2} \ar[l] \ar[r] &
              r_1\ord{2} &
              s_1\ord{2} \ar[l] \ar[r] &
              r_2\ord{2}
            \end{tikzcd}
          \]
  \end{enumerate}
\end{example}

\begin{definition}\label{def:truss-bundle}
  Let $\cat{B}$ and $\cat{C}$ be $\infty$-categories. The space
  $\Truss(\cat{B}, \cat{C})$ of \textit{truss bundles} over $\cat{B}$
  with labels in $\cat{C}$ is the space of diagrams of $\infty$-categories
  of the form
  \begin{equation}\label{eqn:truss-bundle}
    \begin{tikzcd}
      \cat{C} & \cat{E} \ar[d] \ar[r] \ar[l] \pullback & \ETruss \ar[d] \\
      & \cat{B} \ar[r] & \FinOrd
    \end{tikzcd}
  \end{equation}
  where the square is cartesian and $\ETruss \to \FinOrd$ is the universal truss bundle
  from Definition~\ref{def:truss}.

  A functor $\cat{B}' \to \cat{B}$ between the base $\infty$-categories induces a map of spaces $\Truss(\cat{B}, \cat{C}) \to \Truss(\cat{B}', \cat{C})$ by taking the pullback.
A functor $\cat{C} \to \cat{C}'$ between the label $\infty$-categories induces a map of spaces $\Truss(\cat{B}, \cat{C}) \to \Truss(\cat{B}, \cat{C}')$ by postcomposing the labelling functor $\cat{E} \to \cat{C}$ with $\cat{C} \to \cat{C}'$.
  Together this defines a functor
  $\Truss(-, -): \CatInfty^\op \times \CatInfty \to \Space$.
\end{definition}

Formulated differently, given $\infty$-categories $\cat{B}$ and $\cat{C}$, a point in the
space $\Truss(\cat{B}, \cat{C})$ consists of an $\infty$-category $\cat{E}$, functors
$\cat{B} \to \Delta$, $\cat{E} \to \cat{B}$, $\cat{E} \to \ETruss$ and $\cat{E} \to \cat{C}$ together with a witness that the square
(\ref{eqn:truss-bundle}) commutes such that the square is a pullback square.
We will often denote the truss bundle just by the functor $\cat{E} \to \cat{B}$ where the rest of the structure is clear from context.

In particular, since both pullbacks and maps into the terminal category $*$ are defined up to a contractible choice, the space of
truss bundles $\Truss(\cat{B}, *)$ labelled in $*$ is naturally equivalent to the
space of functors $\CatInfty(\cat{B}, \Delta)$ and so the presheaf $\Truss(-, *)$
is represented by the category $\FinOrd$. To show that $\Truss(\cat{B}, \cat{C})$ is
representable for general labelling $\infty$-categories $\cat{C}$ we will use the
theory of exponentiable fibrations as described in~\cite{ayala2017fibrations} for $\infty$-categories in general or~\cite{lurie-ha} for quasicategories in particular, generalising the $1$-categorical Conduché functors.

\begin{lem}
  Let $\cat{B}$ be a poset and $\cat{C}$ an $\infty$-category.
  Then every truss bundle in $\Truss(\cat{B}, \cat{C})$ is equivalent to one of the form $\cat{E} \to \cat{B}$ where $\cat{E}$ is a poset.
\end{lem}

\begin{lem}\label{lem:truss-exponentiable}
  The universal truss bundle $\ETruss \to \FinOrd$ is exponentiable,
  i.e.\ the base change functor
  \[ \CatInfty \downarrow \FinOrd \to \CatInfty \downarrow \ETruss \]
  has a right adjoint.
\end{lem}
\begin{proof}
  We use~\cite[Lemma 1.11.]{ayala2017fibrations} to verify the claim.
  Let $\alpha : \Delta\ord{2} \to \FinOrd$ be a commutative triangle in $\FinOrd$.
  We then need to show that the fibre of the restriction functor
  \begin{equation}\label{eqn:truss-exponentiable-restriction}
    \Fun_{/\FinOrd}(\Delta\ord{2}, \ETruss) \to \Fun_{/\FinOrd}(\Delta\{0 < 2\}, \ETruss)
  \end{equation}
  over any map $h : \Delta\{ 0 < 2 \} \to \ETruss$ is weakly contractible.
  Since $\ETruss \times_\FinOrd \Delta\ord{2}$ is a poset, a factorisation in $\ETruss \times_\FinOrd \Delta\ord{2}$ is uniquely determined by the object it factors through.
  Hence the fibre of (\ref{eqn:truss-exponentiable-restriction}) over any $h$ is isomorphic to a subposet of $\ETruss \times_\Delta \{ \alpha(1) \}$.
  This subposet is weakly contractible by a case distrinction.
\end{proof}

\begin{prop}\label{prop:truss-representable}
  Let $\cat{C}$ be an $\infty$-category. Then the $\infty$-presheaf
  \[ \Truss(-, \cat{C}) : \CatInfty^\op \to \Space \]
  is representable by an $\infty$-category $\Truss(\cat{C})$, i.e.\ there is a natural equivalence
  \[ \Truss(-, \cat{C}) \simeq \CatInfty(-, \Truss(\cat{C})). \]
\end{prop}
\begin{proof}
  Since $\ETruss \to \FinOrd$ is exponentiable, the composite
  \[
    \CatInfty \downarrow \FinOrd \to
    \CatInfty \downarrow \ETruss \to
    \CatInfty
  \]
  of base change and the forgetful functor has a right adjoint
  which sends $\cat{C}$ to an $\infty$-category $\Truss(\cat{C})$ over $\FinOrd$.
  The unstraightening $\GrContr(\Truss(-, \cat{C})) \to \CatInfty$ of the functor $\Truss(-, \cat{C}) \to \Space$ is the $\infty$-category of diagrams of the form (\ref{eqn:truss-bundle}) where the $\infty$-category $\cat{B}$ is allowed to vary.
  The right fibration $\GrContr(\Truss(-, \cat{C})) \to \CatInfty$ then factors as the composition of the right fibration $\GrContr(\Truss(-, \cat{C})) \to {\CatInfty \downarrow \FinOrd}$ which remembers only the map $\cat{B} \to \FinOrd$ followed by the right fibration ${\CatInfty \downarrow \FinOrd} \to \CatInfty$.
  Since for any $\varphi : \cat{B} \to \FinOrd$ the space of
  pullbacks of $\varphi$ along $\ETruss \to \FinOrd$ is
  contractible there is a natural equivalence
  \[
    \GrContr(\Truss(-, \cat{C})) \times_{{\CatInfty \downarrow \FinOrd}} \{\varphi\} \simeq
    \CatInfty(\cat{B} \times_\FinOrd \ETruss, \cat{C}) \simeq
    {(\CatInfty \downarrow \FinOrd)}(\cat{B}, \Truss(\cat{C})).
  \]
  Since this equivalence is natural in $\varphi$, we obtain an equivalence $\GrContr(\Truss(-, \cat{C})) \simeq {(\CatInfty \downarrow \Truss)(\cat{C})}$ over ${\CatInfty \downarrow \FinOrd}$ and by postcomposition with the forgetful functor ${\CatInfty \downarrow \FinOrd} \to \CatInfty$ this equivalence lies over $\CatInfty$.
  By straightening we then obtain a natural equivalence $\Truss(-, \cat{C}) \simeq \CatInfty(-, \Truss(\cat{C}))$.
\end{proof}

Since the $\infty$-presheaf $\Truss(-, \cat{C})$ is natural in the labelling $\infty$-category $\cat{C}$ and $\Truss(\cat{C})$ is defined as the $\infty$-category represented by $\Truss(-, \cat{C})$ it follows that $\Truss(\cat{C})$ itself is natural in $\cat{C}$ and defines an endofunctor $\Truss(-) : \CatInfty \to \CatInfty$.

We have shown that for any $\infty$-category $\cat{C}$ there is an $\infty$-category $\Truss(\cat{C})$ of truss bundles labelled in $\cat{C}$.
In particular $\Truss(\cat{C})$ can itself serve as the $\infty$-category of labels, so that we obtain an $\infty$-category $\Truss(\Truss(\cat{C}))$.
This process can be iterated to obtain $\infty$-categories $\Truss^n(\cat{C})$ for any $n \geq 0$ defined by $\Truss^n(\cat{C}) = \Truss(\Truss^{n - 1}(\cat{C}))$.

\begin{definition}
  Let $n \geq 0$ and $\cat{C}$ an $\infty$-category.
  An \textit{$n$-truss} with labels in $\cat{C}$ is an object of the $\infty$-category $\Truss^n(\cat{C})$.
  A \textit{bordism} of $n$-trusses is a $1$-morphism in $\Truss^n(\cat{C})$.
\end{definition}

The $\infty$-category $\Truss^n(\cat{C})$ admits an alternative characterisation via sequences of truss bundles.

\begin{definition}
  Let $n \geq 0$ and let $\cat{B}, \cat{C}$ be $\infty$-categories. The space
  $\Truss^n(\cat{B}, \cat{C})$ of $n$-fold truss bundles over $\cat{B}$
  with labels in $\cat{C}$ is the space of diagrams of $\infty$-categories of
  the form
  \[
    \begin{tikzcd}
      \cat{C} &
      \cat{E}_n \ar[r, "\pi_n"] \ar[l] &
      \cat{E}_{n - 1} \ar[r, "\pi_{n - 1}"] &
      \cdots \ar[r] &
      \cat{E}_1 \ar[r, "\pi_1"] &
      \cat{E}_0 = \cat{B}
    \end{tikzcd}
  \]
  together with a pullback square of the form
  \[
    \begin{tikzcd}
      \cat{E}_{k + 1} \ar[d, "\pi_k"] \ar[r] \pullback & \ETruss \ar[d] \\
      \cat{E}_k \ar[r] & \FinOrd
    \end{tikzcd}
  \]
  for each $0 \leq k < n$.
  This assembles into a functor $\Truss^n(-, -) : \CatInfty^\op \times \CatInfty \to \Space$ 
  which acts on the base $\infty$-category $\cat{B}$ by pullback of bundles
  and on the labelling $\infty$-category $\cat{C}$ by postcomposition
  as in Definition~\ref{def:truss-bundle}.
\end{definition}

\begin{example}
  Special cases of this definition of $n$-fold truss bundles are
  $\Truss^0(\cat{B}, \cat{C}) \simeq \CatInfty(\cat{B}, \cat{C})$ and
  $\Truss^1(\cat{B}, \cat{C}) \simeq \Truss(\cat{B}, \cat{C})$.
\end{example}

\begin{construction}\label{constr:truss-packing}
  Let $n \geq 0$ and let $\cat{B}, \cat{C}$ be $\infty$-categories.
  The \textit{packing functor}
  \[\Truss^{n + 1}(\cat{B}, \cat{C}) \to \Truss^n(\cat{B}, \Truss(\cat{C})) \]
  sends an $(n + 1)$-fold truss bundle
  \begin{equation}\label{eqn:truss-packing-unpacked}
    \begin{tikzcd}
      \cat{C} &
      \cat{E}_{n + 1} \ar[r, "\pi_{n + 1}"] \ar[l] &
      \cat{E}_{n} \ar[r, "\pi_n"] &
      \cdots \ar[r] &
      \cat{E}_1 \ar[r, "\pi_1"] &
      \cat{E}_0 = \cat{B}
    \end{tikzcd}
  \end{equation}
  to the $n$-fold truss bundle
  \begin{equation}\label{eqn:truss-packing-packed}
    \begin{tikzcd}
      {} &
      \Truss(\cat{C}) &
      \cat{E}_n \ar[r, "\pi_n"] \ar[l] &
      \cdots \ar[r] &
      \cat{E}_1 \ar[r, "\pi_1"] &
      \cat{E}_0 = \cat{B}
    \end{tikzcd}
  \end{equation}
  where $\cat{E}_n \to \Truss(\cat{C})$ is determined by applying the equivalence
  of Proposition~\ref{prop:truss-representable}.
\end{construction}

\begin{prop}
  The packing functor from Construction~\ref{constr:truss-packing} is a natural
  equivalence. In particular the presheaf $\Truss^n(-, \cat{C})$ is
  represented by the $\infty$-category $\Truss^n(\cat{C})$
\end{prop}
\begin{proof}
  Both $\Truss^{n + 1}(\cat{B}, \cat{C})$ and $\Truss^n(\cat{B}, \Truss(\cat{C}))$
  admit a functor to $\Truss^n(\cat{B}, *)$ which truncates (\ref{eqn:truss-packing-unpacked})
  and (\ref{eqn:truss-packing-packed}), respectively, to the unlabelled $n$-fold truss bundle
  \[
    \begin{tikzcd}
      \cat{E}_n \ar[r, "\pi_n"] &
      \cdots \ar[r] &
      \cat{E}_1 \ar[r, "\pi_1"] &
      \cat{E}_0 = \cat{B}
    \end{tikzcd}
  \]
  The truncation functors are right fibrations (via pullback of bundles and precomposition of labelling functors) and commute with the packing functor.
  It hence suffices to show that the packing functor induces an equivalence of the fibres, which follows from Proposition~\ref{prop:truss-representable}.
\end{proof}

The theory of trusses in~\cite{fct} contained maps of (iterated) truss bundles.
These maps are subsumed in the truss bordisms in our formulation of the theory.
While this choice makes a duality theory or simultaneous treatment of open and closed trusses as in~\cite{fct} harder due to the presence of outer face maps,
it allows to treat truss maps and bordisms in a unified manner.
In particular, it allows for the following comparison functor with $\Delta$.

\begin{construction}
  For any $\infty$-category $\cat{C}$, there is a natural fully faithful functor
  \[
    \Delta \times \cat{C} \to \Truss(\cat{C})
  \]
  which sends $(\ord{k}, c)$ to the truss bundle over $\ord{k}$ which is constantly labelled in $c$.
  By induction on $n \geq 0$, this induces a fully faithful functor $\Delta^{\times n} \times \cat{C} \to \Truss^n(\cat{C})$
  from the $n$-fold product of $\Delta$ with $\cat{C}$ into $n$-fold truss bundles labelled in $\cat{C}$.
\end{construction}

In particular for $\cat{C} = *$ we obtain an inclusion $\Delta^{\times n} \to \Truss^n(*)$.
In future work we will use this inclusion
in order to compare a theory of $(\infty, n)$-categories built on trusses with $n$-fold Segal spaces.

\section{Stratified Spaces}\label{sec:stratified-spaces}

Framed combinatorial topology studies stratified spaces which are expressible via meshes.
General poset-stratified spaces, conical stratified spaces or homotopically stratified spaces are too general for our purposes, while the conically smooth stratified spaces of~\cite{ayala2018stratified} are too rigid.
In this section we aim to strike a balance by outlining a notion of stratified spaces that is appropriate for our theory.

In particular, we define an $\infty$-category $\Strat$ of well-behaved stratified spaces represented by a Kan enriched category $\StratPre$,
together with a fully faithful functor $\Exit : \Strat \to \CatInfty$ which sends a stratified space to the $\infty$-category of exit paths.
This will allow us to characterise $\CatInfty$ as a full subcategory of $\PSh(\Strat)$.
The concepts in this section such as poset-stratified spaces and exit path $\infty$-categories are standard; we only claim originality of how we pull the parts together.

\begin{definition}
  The Alexandroff topology on a poset $P$ is the topology on the underlying set
  of $P$ in which a subset $U \subseteq P$ is open if it is upwards closed, i.e.\
  if for every $x \in U$ and $x \leq y$ also $y \in U$.
\end{definition}

For a convenient category of topological spaces, we use the category $\Top$ of numerically generated topological spaces~\cite{dugger-delta, fajstrup2008convenient}.
Equipping posets with the Alexandroff topology induces a fully faithful functor
$\Pos \hookrightarrow \Top$.
This is not true for the more common choice of weakly Hausdorff spaces, as only discrete posets are weakly Hausdorff with when equipped with the Alexandroff topology.
Moreover numerically generated topological spaces form a locally presentable category, which will be convenient for some arguments.

\begin{definition}
  The category of \textit{poset-stratified spaces} is the comma category $\Top \downarrow \Pos$ of the
  inclusion $\Pos \to \Top$. In particular, the objects are maps $X \to P$ from
  a numerically generated topological space into a poset equipped with the Alexandroff
  topology. 
  The preimage of an element $p \in P$ is called a \textit{stratum} of $X$.
  The maps of $\Top \downarrow \Pos$ are commutative squares of the form
  \[
    \begin{tikzcd}
      X \ar[r] \ar[d] & Y \ar[d] \\
      P \ar[r] & Q
    \end{tikzcd}
  \]
\end{definition}

By sending a space $X$ to the unique map $X \to \ord{0}$ we obtain a fully faithful functor $\Top \to \Top \downarrow \Pos$ which we will use to consider unstratified spaces as trivially stratified spaces.
In particular, we obtain trivially stratified spaces for the topological standard simplices $\DeltaTop{n}$,
which we use to enrich $\Top \downarrow \Pos$ in simplicial sets.

\begin{definition}
  The category $\Top \downarrow \Pos$ of poset-stratified spaces can be simplically enriched:
  For poset-stratified spaces
  $X \to P$ and $Y \to Q$ the simplicial set of maps is
  \[ \Map_{\Top \downarrow \Pos}(X \to P, Y \to Q): \ord{n} \mapsto (\Top \downarrow \Pos)(X \times \DeltaTop{n}, Y).\] The composition of maps
  $f : X \times \DeltaTop{n} \to Y$ and $g : Y \times \DeltaTop{n} \to Z$ is defined
  as the map $X \times \DeltaTop{n} \to Z$ which sends $(x, s)$ to
  $g(f(x, s), s)$.
\end{definition}

\begin{lem}
  $\Top \downarrow \Pos$ is tensored and cotensored as a simplicially enriched category.
\end{lem}
\begin{proof}
  $\Top \downarrow \Pos$ is tensored by the definition of the simplicial set of maps.

  Let $X \to P$ be an object of $\Top \downarrow \Pos$ and $A$ a simplicial set. Since
  the tensoring $A \otimes - : \Top \downarrow \Pos \to \Top \downarrow \Pos$ preserves colimits
  and $\Top \downarrow \Pos$ is locally presentable, it has a right adjoint
  $-^A : \Top \downarrow \Pos \to \Top \downarrow \Pos$. Since we have a natural isomorphism
  \[ (\Top \downarrow \Pos)(X \times |A| \times \DeltaTop{n}, Y)
    \cong (\Top \downarrow \Pos)(X \times \DeltaTop{n}, Y^A) \]
  this functor defines a cotensoring.
\end{proof}

\begin{lem}
  $\Top \downarrow \Pos$ is enriched in Kan complexes.
\end{lem}
\begin{proof}
  The Kan fillers can be constructed via any choice of retraction of the topological horn inclusions $\HornTop{n}{k} \hookrightarrow \DeltaTop{n}$ for $0 \leq k \leq n$.
\end{proof}

An important example of stratified spaces are simplicial complexes equipped with their canonical stratification~\cite[Section~A.6]{lurie-ha}.

\begin{definition}
  An abstract simplicial complex $(V, S)$ is a pair of a finite set $V$ of vertices and a
  set $S$ of simplices, where a simplex is a finite subset of $V$.
  Whenever $\sigma \subseteq V$ is a simplex in $S$, any subset $\tau \subseteq \sigma$ is also
  a simplex in $S$.
\end{definition}

An abstract simplicial complex $(V, S)$ can be geometrically realised as a topological space $|(V, S)|$.
For convenience we define $|(V, S)|$ to be the geometric realisation of the nerve of $S$ regarded as a poset; this is homeomorphic to the usual notion of geometric realisation.
Then we obtain the canonical stratification $|(V, S)| \to S$ by sending $x \in |(V, S)|$ to $s \in S$ when $x$ is contained in the interior of a simplex $\sigma : \ord{n} \to S$ where $\sigma(n) = s$.

\begin{definition}
  A poset-stratified space $X \to P$ is \textit{triangulable} if there exists a simplicial complex $(S, V)$
  together with a diagram in $\Top$ of the form
  \[
    \begin{tikzcd}
      {|(V, S)|} \ar[d] \ar[r, "\cong"] & X \ar[d] \\
      S \ar[r] & P
    \end{tikzcd}
  \]
\end{definition}

In this definition we do not require $S \to P$ to be an isomorphism;
in particular we allow the stratification of $X$ to be coarser than the stratification of $|(V, S)|$
by sending different elements of $S$ to the same element in $P$.

\begin{definition}
  Let $n \geq 0$. The stratified $n$-simplex $\DeltaStrat{n}$ is the stratification of the
  topological $n$-simplex $\DeltaTop{n}$ by the map $\DeltaTop{n} \to \ord{n}$ which sends a point
  $(t_0, \ldots, t_n)$ in barycentric coordinates to the maximum $0 \leq i \leq n$ such that $t_i \neq 0$.
  This induces a cosimplicial object $\Delta \to \Top \downarrow \Pos$ in poset-stratified spaces.
\end{definition}

\begin{definition}
  For any poset-stratified space $X \to P$ there is a simplicial set $\Exit(X)$ whose $n$-simplices are
  stratified maps $\DeltaStrat{n} \to X$.
  This defines a functor $\Exit : \Top \downarrow \Pos \to \sSet$.
\end{definition}

\begin{definition}
  A poset-stratified space $X \to P$ is \textit{fibrant} when $\Exit(X)$ is a quasicategory.
\end{definition}

\begin{definition}
  We write $\StratPre$ for the full subcategory of $\Top \downarrow \Pos$ consisting of poset-stratified spaces $X \to P$ that are fibrant, triangulable and have non-empty path-connected strata.  
  Then we let $\Strat$ denote the $\infty$-category arising from $\StratPre$ by Dwyer-Kan localisation at the stratified homotopy equivalences.
\end{definition}

\begin{example}
  Every poset-stratified space that is conical in the sense of~\cite[Definition A.5.5.]{lurie-ha}
  is fibrant by~\cite[Theorem A.6.4.]{lurie-ha}.
  Moreover the canonical stratification of any simplicial complex is fibrant by~\cite[Corollary A.6.9.]{lurie-ha}.
  However, not every triangulable poset-stratified space is fibrant:
  for instance the stratification of the horn $\HornTop{2}{1}$ as a subspace of $\DeltaStrat{2}$ is not fibrant.
\end{example}

The universal property of the Dwyer-Kan localisation states that for every $\infty$-category $\cat{C}$ the space of functors $\Strat \to \cat{C}$
is equivalent to the space of functors $\StratPre \to \cat{C}$ which send stratified homotopy equivalences to equivalences in $\cat{C}$.
In particular, we see that $\Exit : \Top \downarrow \Pos \to \sSet$ descends to a functor of $\infty$-categories $\Exit : \Strat \to \CatInfty$.

We will show that $\Exit : \Strat \to \CatInfty$ is fully faithful.
For this it is convenient to observe that
by~\cite[Example 1.3.4.8.]{lurie-ha} the $\infty$-category $\Strat$ is equivalent to the $\infty$-category represented by $\StratPre$ as a Kan enriched category via the enrichment inherited from $\Top \downarrow \Pos$.
The mapping spaces of the $\infty$-category $\Strat$ therefore are represented by the mapping spaces in the Kan enrichment of $\StratPre$.

\begin{definition}
  There is an adjunction between the comma categories
  \[ \| - \| : \sSet \downarrow \Pos \rightleftarrows \Top \downarrow \Pos : \SingStrat \]
  where $\|X \to P\|$ is the geometric realisation $|X|$ stratified over $P$ via the
  composition $|X| \to |P| \to P$, and $\SingStrat(X \to P)$ is the canonical map of simplicial sets
  $\Exit(X) \to P$.
  For any fixed poset $P$ the adjunction restricts to
  an adjunction
  \[ \| - \|_P : \sSet \downarrow P \rightleftarrows \Top \downarrow P : \SingStrat_P \]
\end{definition}

\begin{lem}
  The adjunctions $\| - \| \vdash \SingStrat$ and $\| - \|_P \vdash \SingStrat_P$
  for every poset $P$ are simplicial.
\end{lem}
\begin{proof}
  For any simplicial set $A$ and object $B \to P$ in $\sSet/\Pos$ we have that
  $\| A \otimes B \| \cong A \otimes \| B \|$ since ordinary geometric realisation
  $|-| : \sSet \to \Top$ preserves products. Hence $\| - \|$ defines a simplicial
  functor which acts on $n$-simplices via
  \[
    \Hom_{\sSet/\Pos}(\Delta\ord{n} \otimes A, B) \to \Hom_{\Top/\Pos}(\| \Delta\ord{n} \otimes A\|, B)
    \to \Hom_{\Top/\Pos}(\Delta\ord{n} \otimes \| A \|, B)
  \]
  It now follows that $\| - \| \vdash \SingStrat$ defines a simplicial adjunction
  by~\cite[Proposition 3.7.10.]{riehl2014categorical}. This simplicial adjunction
  then restricts to a simplicial adjunction $\| - \|_P \vdash \SingStrat_P$ for
  any poset $P$.
\end{proof}

\begin{lem}\label{lem:strat-forget-poset}
  Let $A \to P$ and $B \to Q$ be poset-stratified simplicial sets. When $A \to P$ is
  fibrant in the Joyal-Kan model structure on $\sSet \downarrow P$ and $A$ has path connected
  strata, then the forgetful functor $\sSet \downarrow \Pos \to \sSet$ induces an isomorphism of
  Kan complexes
  \[
    \Map_{\sSet \downarrow \Pos}(A, B) \to \Map_{\sSet}(A, B)
  \]
\end{lem}
\begin{proof}
  Because $A \to P$ is fibrant, $A$ is a quasicategory and $A \to P$ a conservative
  functor.
  Then $P$ is the
  $(-1)$-truncation of the quasicategory $A$ and for every
  stratified simplicial set $B \to Q$ and map $A \to B$ there exists
  a unique map $P \to Q$ which makes the following diagram commute.
  \[
    \begin{tikzcd}
      A \ar[r] \ar[d] & B \ar[d] \\
      P \ar[r, dashed, "\exists!"] & Q
    \end{tikzcd}
  \]
  Now the claim follows since products with $\Delta\ord{n} \to \ord{0}$ preserve fibrancy.
\end{proof}

\begin{prop}
  Let $P$ be a poset. Then there exists a simplicial combinatorial model structure ${(\sSet \downarrow P)}_\textup{JK}$,
  called the Joyal-Kan model structure, in which the cofibrations are the monomorphisms and which satisfies the following property:
  When $X \to P$ is a poset-stratified space then $\SingStrat_P(X)$ is fibrant in ${(\sSet \downarrow P)}_\textup{JK}$ if and only if $\Exit(X)$ is a quasicategory.
\end{prop}
\begin{proof}
  See~\cite{haine2022homotopy}.
\end{proof}

\begin{prop}\label{prop:strat-exit-ff}
  The functor $\Exit : \Strat \to \CatInfty$ is fully faithful.
\end{prop}
\begin{proof}
  Let $X \to P$ be a stratified space. We pick a triangulation
  \[
    \begin{tikzcd}
      {|(V, S)|} \ar[d] \ar[r, "\cong"] & X \ar[d] \\
      S \ar[r] & P
    \end{tikzcd}
  \]
  where $(V, S)$ is an abstract simplicial complex.
  Via the map $S \to P$ we can see $S$ itself as a $P$-stratified simplicial set in $\sSet \downarrow P$.
  Then by~\cite[Lemma 5.3.]{douteau2021homotopy}
  and \cite[Remark 2.51.]{douteau2021homotopy}
  the unit
  $\eta_S : S \to \SingStrat_P(\| S \|_P)$
  of the adjunction ${\| - \|}_P \dashv \SingStrat_P$
  is a weak equivalence in the Joyal-Kan model structure on $\sSet \downarrow P$.
  When $Y \to P$ is a poset-stratified space such that $\SingStrat_P(Y)$ is fibrant,
  then the precomposition map ${\eta_S}^*$ induces a weak homotopy equivalence of Kan complexes in the diagram
  \[
    \begin{tikzcd}
      \Map_{\Top \downarrow P}(\| S \|_P, Y) \ar[swap, dr, "\simeq"] \ar[r, "\SingStrat_P"] &
      \Map_{\sSet \downarrow P}(\SingStrat_P(\| S \|_P), \SingStrat_P(Y)) \ar[d, "{\eta_S}^*"] \\
      & \Map_{\sSet \downarrow P}(S, \SingStrat_P(Y))
    \end{tikzcd}
  \]
  Then by 2-out-of-3 it follows that $\Sing_P$ induces a weak homotopy equivalence  
  \[ \Map_{\Top \downarrow P}(\| S \|_P, Y) \to \Map_{\sSet \downarrow P}(\SingStrat_P(\| S \|_P), \SingStrat_P(Y)). \]
  Since $| S | \cong |(V, S)| \cong X$ it then follows that $\Sing_P$ induces a weak homotopy equivalence
  \[ \Map_{\Top \downarrow P}(X, Y) \to \Map_{\sSet \downarrow P}(\SingStrat_P(X), \SingStrat_P(Y)). \]

  Now let $X \to P$ and $Y \to Q$ be stratified spaces.
  Every stratified map $f : X \to Y$ with underlying map of posets $\alpha : P \to Q$
  factors through the pullback
  \[
    \begin{tikzcd}
      X \ar[r] \ar[d] & \alpha^* Y \ar[r] \ar[d] \pullback & Y \ar[d] \\
      P \ar[r, "\id"] & P \ar[r, "\alpha"] & Q
    \end{tikzcd}
  \]
  In particular there is a commutative square of Kan complexes
  \[
    \begin{tikzcd}
      \Map_{\Top \downarrow \Pos}(X, Y) \ar[r, "\SingStrat_*"] \ar[d, "\simeq"] &
      \Map_{\sSet \downarrow \Pos}(\SingStrat(X), \SingStrat(Y)) \ar[d, "\simeq"] \\
      \coprod_{\alpha : P \to Q} \Map_{\Top \downarrow P}(X, \alpha^* Y) \ar[r] &
      \coprod_{\alpha : P \to Q} \Map_{\sSet \downarrow P}(\SingStrat_P(X), \SingStrat_P(\alpha^* Y))
    \end{tikzcd}
  \]
  Since each $\alpha^* Y$ is fibrant, the map on the bottom is a weak homotopy equivalence,
  and hence so is the map on the top. So $\SingStrat$ restricts to a fully
  faithful simplicial functor $\StratPre \to \sSet \downarrow \Pos$ and by Lemma~\ref{lem:strat-forget-poset}
  to a fully faithful simplicial functor $\StratPre \to \sSet$ by postcomposition with the
  forgetful functor $\sSet \downarrow \Pos \to \sSet$. But this functor is $\Exit$.
\end{proof}

We can now see $\CatInfty$ as a full subcategory of $\PSh(\Strat)$.
This is a similar idea to the striation sheaves of~\cite{ayala2018stratified}.
We refrain however from characterising those presheaves on $\Strat$ which land in the essential image
since we do not need such a characterisation for the purposes of this paper.

\begin{prop}
  The functor $\CatInfty \to \PSh(\Strat)$ which sends an $\infty$-category
  $\cat{C}$ to the space-valued presheaf
  $\CatInfty(\Exit(-), \cat{C})$ is fully faithful.
\end{prop}
\begin{proof}
  We have a commutative triangle of functors
  \[
    \begin{tikzcd}[column sep = 1.5cm]
      \Delta \ar[r, "\DeltaStrat{-}"] \ar[dr, hookrightarrow] &
      \Strat \ar[d, "\Exit", hookrightarrow] \\
      & \CatInfty
    \end{tikzcd}
  \]
  and so it follows that $\DeltaStrat{-}$ is fully faithful. This implies
  that the right adjoint $\PSh(\Delta) \to \PSh(\Strat)$ of
  the restriction functor $\PSh(\Strat) \to \PSh(\Delta)$
  is fully faithful as well. $\CatInfty$ is
  a full subcategory of $\PSh(\Delta)$ and by inspection the right adjoint
  functor is the functor from the claim.
\end{proof}

By Proposition~\ref{prop:strat-exit-ff} the representable presheaf $\Strat(-, X)$
of a stratified space $X$
is equivalent to the presheaf $\CatInfty(\Exit(-), \Exit(X))$. Hence we have
that $\Strat(-, X)$ represents the $\infty$-category $\Exit(X)$.

\section{Meshes}

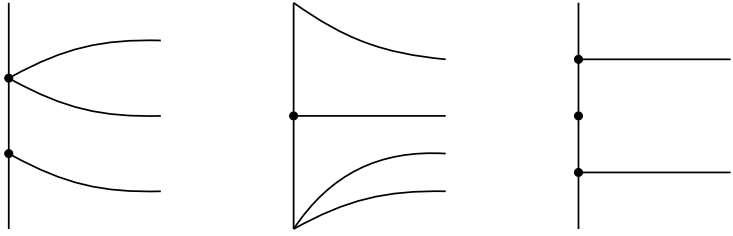
\begin{figure}
  \centering
  \begin{tikzpicture}
	\begin{pgfonlayer}{nodelayer}
		\node [style=none] (0) at (-2, 3) {};
		\node [style=none] (1) at (-2, -3) {};
		\node [style=singularity] (2) at (-2, -1) {};
		\node [style=singularity] (3) at (-2, 1) {};
		\node [style=none] (4) at (2, 2) {};
		\node [style=none] (5) at (2, -2) {};
		\node [style=none] (6) at (2, 0) {};
	\end{pgfonlayer}
	\begin{pgfonlayer}{edgelayer}
		\draw [style=edge] (0.center) to (1.center);
		\draw [style=edge, bend left=15] (3) to (4.center);
		\draw [style=edge, bend right=15] (3) to (6.center);
		\draw [style=edge, bend right=15] (2) to (5.center);
	\end{pgfonlayer}
\end{tikzpicture}
  \qquad\qquad
  \begin{tikzpicture}
	\begin{pgfonlayer}{nodelayer}
		\node [style=none] (0) at (-2, 3) {};
		\node [style=none] (1) at (-2, -3) {};
		\node [style=singularity] (2) at (-2, 0) {};
		\node [style=none] (3) at (2, 0) {};
		\node [style=none] (4) at (2, 1.5) {};
		\node [style=none] (5) at (2, -1) {};
		\node [style=none] (6) at (2, -2) {};
	\end{pgfonlayer}
	\begin{pgfonlayer}{edgelayer}
		\draw [style=edge] (1.center) to (0.center);
		\draw [style=edge] (2) to (3.center);
		\draw [style=edge, bend left=15] (4.center) to (0.center);
		\draw [style=edge, bend right] (5.center) to (1.center);
		\draw [style=edge, bend right=15] (6.center) to (1.center);
	\end{pgfonlayer}
\end{tikzpicture}
  \qquad\qquad
  \begin{tikzpicture}
	\begin{pgfonlayer}{nodelayer}
		\node [style=none] (0) at (-2, 3) {};
		\node [style=none] (1) at (-2, -3) {};
		\node [style=singularity] (2) at (-2, 0) {};
		\node [style=singularity] (3) at (-2, -1.5) {};
		\node [style=singularity] (4) at (-2, 1.5) {};
		\node [style=none] (5) at (2, 1.5) {};
		\node [style=none] (6) at (2, -1.5) {};
	\end{pgfonlayer}
	\begin{pgfonlayer}{edgelayer}
		\draw [style=edge] (1.center) to (0.center);
		\draw [style=edge] (4) to (5.center);
		\draw [style=edge] (3) to (6.center);
	\end{pgfonlayer}
\end{tikzpicture}
  \caption{Mesh bundles over $\DeltaStrat{1}$ which represent the inner face map $(0 \mapsto 0, 1 \mapsto 1, 2 \mapsto 3) : \ord{2} \to \ord{3}$,
  the outer face map $(0 \mapsto 2, 1 \mapsto 3) : \ord{1} \to \ord{4}$ and the degeneracy map $(0 \mapsto 0, 1 \mapsto 1, 2 \mapsto 1, 3 \mapsto 2) : \ord{3} \to \ord{2}$.
  The vertical line is the fibre over the $0$-stratum of $\DeltaStrat{1}$.
  The top and two bottom lines for the outer face map approach the boundary of the open interval $(-1, 1)$ as they approach the left side of the diagram.
\label{fig:mesh-implements-delta}}
\end{figure}

Meshes are a geometric version of trusses that augment truss bundles with a choice of coordinates.
Based on meshes, we can build a theory of tame stratified spaces which is equivalent to the purely combinatorial theory of trusses.

\begin{definition}
  A \textit{$1$-mesh} is a finite stratification of the open interval $\IntOpen$ such
  that every stratum is either an isolated point or an open subinterval.
  A point in a
  $1$-mesh is \textit{singular} if it is an isolated point and \textit{regular}
  otherwise.
\end{definition}

A $1$-mesh should be seen as representing the finite ordinal $\ord{n}$ in which every element $i \in \ord{n}$ is a regular stratum of the mesh.
The exit path $\infty$-category of a $1$-mesh is a $1$-truss, with regular and singular strata corresponding to the regular and singular elements of the truss.

\begin{definition}\label{def:mesh-bundle}
  A \textit{mesh bundle} over a stratified space $B$ is a map of poset-stratified spaces $p : E \to B$ which satisfies the following properties:
  \begin{enumerate}
    \item The underlying map of spaces of $p$ is the projection $B \times \IntOpen \to B$.
    \item $p$ restricts to a trivial bundle $S \times \IntOpen \to S$ of spaces for each stratum $S \subseteq B$.
    \item For each $b \in B$ the fibre $p^{-1}(b)$ is a $1$-mesh.
    \item The regular points of all fibres form an open subset of $B \times \IntOpen$.
  \end{enumerate}
\end{definition}

Mesh bundles over the stratified $1$-simplex $\DeltaStrat{1}$ correspond to order-preserving maps between the ordinals corresponding to the $1$-meshes over the endpoints.
This is illustrated in Figure~\ref{fig:mesh-implements-delta}.

\begin{remark}
  Replacing the open interval $(-1, 1)$ by any other open interval in $\R$ would yield an equivalent theory of meshes.
  When meshes are applied to characterise manifolds or stratified spaces embedded into $\R^n$ it is most natural to use $\R$ itself.
  For the purposes of this paper we have chosen an interval of finite length to simplify arguments based on linear interpolation.
\end{remark}

We note that we defined mesh bundles $p : E \to B$ over stratified spaces $B$ but merely required $E$ to be a poset-stratified space.
Lemma~\ref{lem:mesh-construct} below shows that $E$ has to be triangulable,
and via Proposition~\ref{prop:mesh-universal} we will see that $E$ is in fact a stratified space.

\begin{lem}
  Mesh bundles are closed under pullback.
\end{lem}
\begin{proof}
  Let $p : E \to B$ be a mesh bundle and $f : B' \to B$ a map of stratified spaces.
  The first three properties of Definition~\ref{def:mesh-bundle} are preserved by pullback by the pasting lemma and since the forgetful functor $\StratPre \to \Top$ preserves pullbacks.
  A point in $E \times_{B} B'$ over $b' \in B$ is regular if and only if it is regular in the fibre $(E \times_B B') \times_{B'} \{ b' \} \cong E \times_B \{ f(b') \}$.
  It follows that the set of regular points in $E \times_B B'$ is the preimage of the regular points in $E$ via the induced continuous map $E \times_B B' \to E$ and thus it is open.
\end{proof}

As mentioned above, the fibre of a mesh bundle $p : E \to B$ over any given point is determined up to a contractible choice of coordinates within the interval $\IntOpen$ by the number of its regular strata.
We show that this data fits together coherently in the form of a functor $\reg_p : \Exit(B) \to \FinOrd$. 
Similarly, we can obtain a functor $\sing_p : \Exit(B)^\op \to \StrictInt$ characterised by the singular strata of the fibres.
Moreover, any such functor gives rise to a mesh bundle which is determined up to a contractible choice.
In this sense, any mesh bundle is determined by combinatorial data.

\begin{lem}\label{lem:mesh-reg-monotone}
  Let $p : E \to B$ be a mesh bundle. Then there exists a unique well-defined functor
  \[ \reg_p : \Exit(B) \to \FinOrd \] which satisfies the following properties:
  \begin{enumerate}
    \item When $b \in B$ is a point in the base space then $\reg_p(b)$
          is the set of regular strata in the fibre $p^{-1}(b)$ with a total
          order induced by the canonical order on $\IntOpen$.
    \item When $\gamma : \DeltaStrat{1} \to B$ is an exit path in the base
          space, then $\reg_p(\gamma)$ sends a regular stratum $R_0$ over $\gamma(0)$
          to a regular stratum $R_1$ over $\gamma(1)$ when there exists a lift
          $\hat{\gamma} : \DeltaStrat{1} \to E$ of $\gamma$ which starts in
          $R_0$ and ends in $R_1$.
  \end{enumerate}
\end{lem}
\begin{proof}
  To show that $\reg_p$ is well-defined as a map of sets we first show that the induced functor $\Exit(E) \to \Exit(B)$ restricts to a discrete fibration on the regular objects in $\Exit(E)$.

  Let $\gamma : \DeltaStrat{1} \to B$ be an elementary exit path in the base space and $e \in E$ a regular element such that $p(e) = \gamma(0)$.
  Since $E \cong B \times \IntOpen$ as spaces, we can write $e$ as the product $(\gamma(0), h)$.
  Since the regular
  points of $E$ form an open subset,
  there exists a $\varepsilon > 0$ such that the map $\hat{\gamma} : [0, \varepsilon] \to E$ defined by $t \mapsto (\gamma(t), h)$ is an exit path in $E$ lifting the restricted path $\gamma \upharpoonleft [0, \varepsilon]$.
  We can then extend $\hat{\gamma}$ to an exit path $\DeltaStrat{1} \to E$ lifting the entirety of $\gamma$ by staying within the regular stratum containing $\hat{\gamma}(\varepsilon)$.

  Suppose that $\hat{\gamma}_1, \hat{\gamma}_2$ are two lifts of $\gamma$ starting in a regular element $e \in E$ above $\gamma(0)$.
  Since the regular points of $E$ are an open subset there is an open neighbourhood $U \subseteq E$ of $e$ which does not intersect any singular element.
  In particular when $t > 0$ is small enough such that the path from $\hat{\gamma}_1(t)$ to $\hat{\gamma}_2(t)$ in the fibre over $\gamma(t)$ is contained in $U$, then there is no singular stratum between $\hat{\gamma}_1(t)$ and $\hat{\gamma}_2(t)$ and so both must lie in the same regular stratum.
  But then $\hat{\gamma}_1(1)$ and $\hat{\gamma}_2(1)$ must also lie in the same regular stratum.

  To show that the functor preserves the ordering let $\gamma: \DeltaStrat{1} \to B$ be an elementary exit path, $e_1, e_2$ regular points with $e_1 \leq e_2$ and
  $p(e_1) = p(e_2) = \gamma(0)$ and let $\hat{\gamma}_1$, $\hat{\gamma}_2$ be
  lifts of $\gamma$ starting at $e_1$ and $e_2$.
  Suppose that $\hat{\gamma}_1(1) > \hat{\gamma}_2(1)$.
  Then by continuity there exists a $0 < t < 1$ such that $\hat{\gamma}_1(t) = \hat{\gamma}_2(t)$ and thus $\hat{\gamma}_1(1)$ and $\hat{\gamma}_2(1)$ are in the same regular stratum.
\end{proof}

Similarly, each mesh bundle $p : E \to B$ determines a contravariant functor into $\StrictInt$
based on its singular strata. To make this precise, we need to define the compactification
of $p$.

\begin{definition}
  Let $p : E \to B$ be a mesh bundle.
  The \textit{compactification} of $p$ is the stratified bundle $\bar{p} : \bar{E} \to B$ with underlying map of spaces $B \times [-1, 1] \to B$ such that for each $b \in B$ the fibre $\bar{p}^{-1}(b)$ is obtained by padding the fibre $p^{-1}(b)$ with two $0$-strata at the endpoints.
\end{definition}

\begin{lem}\label{lem:mesh-sing-monotone}
  Let $p : E \to B$ be a mesh bundle. Then there exists a unique well-defined functor
  \[ \sing_p : {\Exit(B)}^\op \to \StrictInt \] which satisfies the following properties:
  \begin{enumerate}
    \item When $b \in B$ is a point in the base space then $\sing_p(b)$
          is the set of singular strata in the fibre $\bar{p}^{-1}(b)$ of
          the compactification with a total
          order induced by the canonical order on $[-1, 1]$.
    \item When $\gamma : \DeltaStrat{1} \to B$ is an exit path in the base
          space, then $\sing_p(\gamma)$ sends a singular stratum $S_1$ over $\gamma(1)$
          to a singular stratum $S_0$ over $\gamma(0)$ when there exists a lift
          $\hat{\gamma} : \DeltaStrat{1} \to \bar{E}$ of $\gamma$ which starts in
          $S_0$ and ends in $S_1$.
  \end{enumerate}
\end{lem}
\begin{proof}
  To show that $\sing_p$ is well-defined as a map of sets we first show that the induced functor $\Exit(E) \to \Exit(B)$ restricts to a discrete fibration on the singular objects in $\Exit(E)$.

  Let $\gamma : \DeltaStrat{1} \to B$ be an elementary exit path in the base space
  and $e \in E$ a singular element such that
  $p(e) = \gamma(1)$.
  By using the local trivialisation, there is a unique lift of $\gamma$
  restricted to $\DeltaStrat{1} \cap (0, 1]$. Since the set of
  singular points is closed in $B \times [0, 1]$
  we can uniquely extend the lift
  by continuity to a lift $\hat{\gamma} : \DeltaStrat{1} \to \sing(E)$.

  To show that the functor preserves the ordering
  let $\gamma: \DeltaStrat{1} \to B$ be an elementary exit path,
  $e_1, e_2$ singular points with $e_1 \leq e_2$ and
  $p(e_1) = p(e_2) = \gamma(1)$ and let $\hat{\gamma}_1$, $\hat{\gamma}_2$ be
  the unique lifts of $\gamma$ ending at $e_1$ and $e_2$.
  Suppose that $\hat{\gamma}_1(0) > \hat{\gamma}_2(0)$.
  By continuity there
  must be a $0 < t < 1$ such that $\hat{\gamma}_1(t) = \hat{\gamma}_2(t)$.
  Because the lifts are stratum preserving, this implies that $e_1 = e_2$
  and leads to a contradiction. Hence $\hat{\gamma}_1(0) \leq \hat{\gamma}_2(0)$.
\end{proof}

\begin{lem}
  Let $p : E \to B$ be a mesh bundle. Then the following diagram commutes:
  \[
    \begin{tikzcd}[row sep = 0.1cm]
      & \FinOrd \ar[dd, "\cong"] \\
      \Exit(B) \ar[ur, "\reg_p"] \ar[swap, dr, "\sing_p"] \\
      & \StrictInt^\op
    \end{tikzcd}
  \]
  where $\FinOrd \cong \StrictInt^\op$ is the canonical duality between the category of finite non-empty ordinals and the category of strict intervals.
\end{lem}

\begin{lem}\label{lem:mesh-construct}
  Let $B$ be a stratified space and $f : \Exit(B)^\op \to \StrictInt$ a functor. Then
  there exists a mesh bundle $p : E \to B$ such that $\sing_p = f$. Moreover if
  $B$ is PL then $E$ can be chosen to be PL as well.
\end{lem}
\begin{proof}
  We pick a triangulation of $B$ that is compatible with the stratification.
  When $B$ is PL then we can pick the triangulation such that each simplex
  $\DeltaStrat{k} \hookrightarrow B$ of the triangulation is linear. We
  then construct a stratified bundle $\bar{p} : \bar{E} \to B$ such that the
  fibre-wise interior is a mesh bundle $p : E \to B$.

  For every vertex $x$ of the triangulation we choose a strictly
  monotone endpoint preserving map $f(x) \to [0, 1]$. This determines
  the heights of the singular strata over $x$.

  Let $k > 0$ and assume by induction that we have already constructed the extended mesh bundle
  over the $(k - 1)$-skeleton of the triangulation of $B$.
  Let $\sigma : \DeltaStrat{k} \hookrightarrow B$ be a stratified simplex of
  the triangulation, then
  for every $0 \leq i \leq k$ the functor $f(\sigma) : \ord{k}^\op \to \StrictInt$
  induces a map $f_i : f(k) \to f(i)$.
  For any point $x$ in the interior of $\DeltaStrat{k}$
  we let the height of the singular stratum over $\sigma(x)$ corresponding
  to some $s \in f(\sigma(x))$ be the convex combination of the heights of
  the singular strata $f_i(s)$ over the vertices of $\sigma$, with coefficients
  taken from the barycentric coordinates of $x$.
\end{proof}

\begin{construction}\label{constr:mesh-base}
  For a stratified space $B$ we denote by $\BMesh(B)$ the simplicial set whose
  $n$-simplices consist of mesh bundles of the form $p : E \to B \times
    \DeltaTop{n}$.
  By pullback along the base space $B$
  the spaces $\BMesh(B)$ arrange to a functor $\BMesh(-) : \Strat^\op \to \sSet$.
\end{construction}

The following Proposition shows that $\BMesh$ is a geometrical version of $\Delta$ in which each finite non-empty ordinal is equipped with coordinates and order-preserving maps are realised as bordisms as in Figure~\ref{fig:mesh-implements-delta}.
The space of choices for these coordinates is contractible and hence $\BMesh$ is equivalent to $\Delta$ as an $\infty$-category.

\begin{prop}\label{prop:mesh-classify}
  The functor $\BMesh : \Strat^\op \to \sSet$ from
  Construction~\ref{constr:mesh-base} factors through the inclusion
  $\Space \hookrightarrow \sSet$ and represents an $\infty$-category equivalent
  to $\FinOrd$.
\end{prop}
\begin{proof}
  We show that for every stratified space $B$ the map $\BMesh(B) \to
    \sSet(\Exit(B), \FinOrd)$ of simplicial sets which assigns
  a mesh bundle $p: E \to B \times \DeltaTop{n}$ the induced map
  $\reg_p : \Exit(B) \times \Delta\ord{n} \to \FinOrd$
  is a trivial Kan fibration.
  Let $k \geq 0$ and consider a lifting problem
  \[
    \begin{tikzcd}
      \partial \Delta\ord{k} \ar[d] \ar[r] & \BMesh(B) \ar[d] \\
      \Delta\ord{k} \ar[r] \ar[ur, dashed] & \sSet(\Exit(B), \FinOrd)
    \end{tikzcd}
  \]
  When $k = 0$ this follows from Lemma~\ref{lem:mesh-construct}. For $k > 0$ we have
  a mesh bundle over $B \times \partial \DeltaTop{k}$ which we can extend to a mesh bundle
  over $B \times \DeltaTop{k}$ via linear interpolation of the position of the singular strata.
\end{proof}

For every mesh bundle $p : E \to B$ the equivalence $\BMesh(B) \simeq \CatInfty(\Exit(B), \BMesh)$
yields a functor $\Exit(B) \to \BMesh$ which in components sends a stratified map
$A \to B$ to the pullback mesh bundle $E \times_B A \to A$.

\section{Labelled Meshes}

Since mesh bundles on a stratified space $B$ are determined by functors $\Exit(B) \to \FinOrd$, there is an equivalence between mesh bundles on $B$ and unlabelled truss bundles on $\Exit(B)$.
We will see that for any mesh bundle $p : E \to B$, also $E$ is a stratified space and so $\Exit(E)$ is an $\infty$-category.
Then by equipping mesh bundles $p : E \to B$ with labelling functors $\Exit(E) \to \cat{C}$ in an $\infty$-category $\cat{C}$, we can extend the equivalence between meshes and trusses to the labelled case.

\begin{definition}
  Let $B$ be a stratified space. Then $\EMesh(B)$ is the simplicial set whose
  $n$-simplices consist of mesh bundles of the form $p : E \to B \times \DeltaTop{n}$ together with a section $B \times \DeltaTop{n} \to E$.
  By pullback along the base
  space the spaces $\EMesh(B)$ arrange into a functor $\EMesh(B) : \Strat^\op \to \sSet$.
\end{definition}

\begin{lem}\label{lem:mesh-forget-section-kan}
  Let $B$ be a stratified space. Then
  the canonical map
  \[ \EMesh(B) \to \BMesh(B) \]
  which forgets the section is a Kan fibration. In particular $\EMesh(-)$ defines
  a functor $\Strat^\op \to \Space$.
\end{lem}
\begin{proof}
  We need to verify that for every solid diagram of the form
  \[
    \begin{tikzcd}
      B \times \HornTop{n}{i} \ar[r] \ar[d, hookrightarrow] & E \ar[d, "p"] \\
      B \times \DeltaTop{n} \ar[r, "\id", swap] \ar[ur, dashed] & B \times \DeltaTop{n}
    \end{tikzcd}
  \]
  where $p$ is a mesh bundle, there exists a dashed filler.
  The filler can be constructed by fixing a triangulation of $B$ compatible
  with the stratification, and then proceeding by induction on the skeleta.
\end{proof}

When $p : E \to B$ is a mesh bundle then the pullback $E \times_B E \to E$ of $p$
along itself is a mesh bundle and admits a canonical section $E \to E \times_B E$.
This defines an element of $\EMesh(E)$ and hence a functor $\Exit(E) \to \EMesh$.

\begin{lem}\label{lem:mesh-to-truss-1}
  Let $p : E \to B$ be a mesh bundle together with a section $s : B \to E$.
  Then there exists a unique well defined functor $f : \Exit(B) \to \ETruss$
  which fits into the diagram
  \[
    \begin{tikzcd}
      & \ETruss \ar[d] \\
      \Exit(B) \ar[r, "\reg_p", swap] \ar[ur, "f", dashed] & \FinOrd
    \end{tikzcd}
  \]
  and is defined on objects $b \in B$ as follows:
  \begin{enumerate}
    \item $f(b) = r_i \ord{n}$ when $s(b)$ is the $i$th regular stratum above $b$
          and $\reg_p(b) = \ord{n}$.
    \item $f(b) = s_i \ord{n}$ when $s(b)$ is the $i$th singular stratum above $b$
          and $\reg_p(b) = \ord{n}$.
  \end{enumerate}
\end{lem}
\begin{proof}
  Every functor $\ord{n} \to \ETruss$ uniquely factors through the pullback:
  \[
    \begin{tikzcd}
      \ord{n} \ar[ddr, bend right, swap, "\id"] \ar[drr, bend left = 15] \ar[dr, dashed] \\
      & \ETruss \times_\FinOrd \ord{n} \ar[r] \ar[d] \pullback & \ETruss \ar[d] \\
      & \ord{n} \ar[r] & \FinOrd
    \end{tikzcd}
  \]
  Since $\ord{n}$ is a poset so is $\ETruss \times_\FinOrd \ord{n}$. It follows
  that the functor $f$ is uniquely determined by its action on
  objects and to show that it exists it suffices to demonstrate that when
  $\gamma : \DeltaStrat{1} \to B$ is an exit path there is a map
  $f(\gamma(0)) \to f(\gamma(1))$ in $\ETruss$ over the order-preserving map $\reg_p(\gamma)$.

  \begin{enumerate}
    \item Suppose $f(\gamma(0)) = r_i \ord{n}$ and $f(\gamma(1)) = r_j \ord{m}$.
          Since $s \circ \gamma$ is a lift of $\gamma$ to an exit path in $E$ we have that
          $\reg_p(i) = j$ by Lemma~\ref{lem:mesh-reg-monotone} and so there exists a map
          $r_i \ord{n} \to r_j \ord{m}$ over $\reg_p(\gamma)$ in $\ETruss$.
    \item Suppose $f(\gamma(0)) = s_i \ord{n}$ and $f(\gamma(1)) = s_j \ord{m}$.
          Then by Lemma~\ref{lem:mesh-sing-monotone} we have $\sing_p(j + 1) = i + 1$
          which implies $\reg_p(i) \leq j < \reg_p(i + 1)$ through the duality
          $\StrictInt^\op \cong \Delta$.
    \item Suppose $f(\gamma(0)) = s_i \ord{n}$ and $f(\gamma(1)) = r_j \ord{m}$.
          Then $\reg_p(i) \leq j \leq \reg_p(i + 1)$ since exit paths starting at $\gamma(0)$ can reach at most as high or low as exit paths starting in the regular strata directly above or below $\gamma(0)$.\qedhere
  \end{enumerate}
\end{proof}

\begin{prop}
  The functors in Lemma~\ref{lem:mesh-to-truss-1} assemble into a equivalence of spaces
  \[ \EMesh(B) \to \CatInfty(\Exit(B), \ETruss) \]
  that is natural in $B \in \Strat$.
  In particular $\EMesh$ represents an $\infty$-category that is equivalent to $\ETruss$.
  Moreover this functor fits into a square of spaces
  \[
    \begin{tikzcd}
      \EMesh(B) \ar[d] \ar[r] & \CatInfty(\Exit(B), \ETruss) \ar[d] \\
      \BMesh(B) \ar[r] & \CatInfty(\Exit(B), \FinOrd)
    \end{tikzcd}
  \]
\end{prop}
\begin{proof}
  We begin by constructing the map of simplicial sets
  \[ \EMesh(B) \to \CatInfty(\Exit(B), \ETruss). \]
  An $n$-simplex of $\EMesh(B)$ consists of a mesh bundle $p : E \to B \times \DeltaTop{n}$ together with a section $s : B \times \DeltaTop{n} \to E$.
  Then Lemma~\ref{lem:mesh-to-truss-1} determines a functor $\Exit(B \times \DeltaTop{n}) \to \ETruss$.
  By precomposition with the natural equivalence $\Exit(B) \times E\ord{n} \to \Exit(B \times \DeltaTop{n})$ we then obtain an $n$-simplex of $\CatInfty(\Exit(B), \ETruss)$.
  This assignment of simplices then extends to a map of simplicial sets.

  For naturality in the stratified space $B$ observe that the functor constructed in Lemma~\ref{lem:mesh-to-truss-1} is determined by the fibres and thus behaves well under pullback of bundles.

  By Proposition~\ref{prop:mesh-classify} the bottom horizontal map in the square is an equivalence
  and by Lemma~\ref{lem:mesh-forget-section-kan} the left vertical map is a Kan fibration.
  Since $\ETruss \to \FinOrd$ is an isofibration, the vertical map on the right is also a Kan fibration.
  Hence it suffices to verify that the top horizontal map induces an equivalence of fibres.

  Let $p : E \to B$ be a mesh bundle.
  The fibre of $\EMesh(B) \to \BMesh(B)$ over $p$ has contractible connected components, and so it suffices that the horizontal map induces a bijection of connected components.

  For surjectivity suppose that we have a map $f : \Exit(B) \to \ETruss$ which fits into the diagram
  \[
    \begin{tikzcd}
      & \ETruss \ar[d] \\
      \Exit(B) \ar[r, "\reg_p", swap] \ar[ur, "f"] & \FinOrd
    \end{tikzcd}
  \]
  Pick a triangulation of $B$ which is compatible with its stratification.
  For every vertex $x$ of the triangulation, we define $s(x)$ to be the $i$th singular point in $p^{-1}(x)$ when $f(x) = s_i\ord{n}$ or an arbitrary point in the $i$th regular stratum in $p^{-1}(x)$ when $f(x) = r_i\ord{n}$.
  We can then extend $s$ to a section $s : B \to E$ of $p$ by linearly interpolating.

  For injectivity suppose that we have two sections $s_1 : B \to E$ and $s_2 : B \to E$ which determine the same map $f : \Exit(B) \to \ETruss$.
  Then for every $x \in B$ the points $s_1(x)$ and $s_2(x)$ lie in the same stratum of $E$ and since the strata of mesh bundles are contractible it follows that $s_1$ and $s_2$ are homotopic.
\end{proof}

\begin{prop}\label{prop:mesh-universal}
  Let $p : E \to B$ be a mesh bundle.
  Then $E$ is a stratified space, the square of $\infty$-categories
  \[
    \begin{tikzcd}[column sep = 2cm]
      \Exit(E) \ar[d, "p_*"] \ar[r, "(E \times_B E \leftrightarrows E)"] & \EMesh \ar[d] \\
      \Exit(B) \ar[r, "(E \overset{p}{\to} B)"] & \BMesh
    \end{tikzcd}
  \]
  commutes and is cartesian.
\end{prop}
\begin{proof}
  We begin by considering the square of simplicial spaces
  \begin{equation}\label{eq:mesh-universal-simplicial-spaces}
    \begin{tikzcd}[column sep = 2cm]
      \Map_{\Top \downarrow \Pos}(\DeltaStrat{-}, E) \ar[d, "p_*"] \ar[r] & \EMesh(\DeltaStrat{-}) \ar[d] \\
      \Map_{\Top \downarrow \Pos}(\DeltaStrat{-}, B) \ar[r] & \BMesh(\DeltaStrat{-})
    \end{tikzcd}
  \end{equation}
  This square commutes since both sides send a stratified map
  $\DeltaStrat{k} \times \DeltaTop{n} \to E$ to the left mesh bundle in the
  diagram of poset-stratified spaces
  \[
    \begin{tikzcd}
      E \times_B (\DeltaStrat{k} \times \DeltaTop{n}) \ar[r] \ar[d] \pullback &
      E \times_B E \ar[d] \ar[r] \pullback &
      E \ar[d, "p"] \\
      \DeltaStrat{k} \times \DeltaTop{n} \ar[r] & E \ar[r, "p", swap] &
      B
    \end{tikzcd}
  \]
  in which both squares are pullbacks. 
  To show that the square (\ref{eq:mesh-universal-simplicial-spaces}) is cartesian, it suffices to show that it is cartesian pointwise for any $\ord{k} \in \Delta$.
  Since $\EMesh(\DeltaStrat{k}) \to \BMesh(\DeltaStrat{k})$ is a Kan fibration,
  the pullback of spaces is represented by an ordinary pullback of simplicial sets.
  When $\ord{n} \in \FinOrd$ then by the universal property of the pullback the two dashed
  maps in the following diagram are uniquely determined by each other.
  \[
    \begin{tikzcd}
      \DeltaStrat{k} \times \DeltaTop{n} \ar[ddr, bend right, swap, "\id"]
      \ar[drr, bend left = 15, dashed] \ar[dashed, dr] \\
      & (\DeltaStrat{k} \times \DeltaTop{n}) \times_B E \ar[r] \ar[d] \pullback & E \ar[d, "p"] \\
      & \DeltaStrat{k} \times \DeltaTop{n} \ar[r] & B
    \end{tikzcd}
  \]

  Now since $\EMesh$ and $\BMesh$ are $\infty$-categories and $E$ is fibrant,
  the square (\ref{eq:mesh-universal-simplicial-spaces}) exhibits
  $\Map_{\Top \downarrow \Pos}(\DeltaStrat{-}, E)$ as a pullback of complete Segal spaces
  and hence as a complete Segal space itself.
  It follows that $\Exit(E)$ is a quasicategory and so $E$ is fibrant.
  By Lemma~\ref{lem:mesh-construct} we also know that $E$ is triangulable.
  Hence $E$ is a stratified space.
\end{proof}

\begin{definition}
  Let $B$ be a stratified space and $\cat{C}$ an $\infty$-category. Then $\Mesh(B, \cat{C})$ is the simplicial set whose
  $n$-simplices consist of mesh bundles of the form $p : E \to B \times
    \DeltaTop{n}$ together with a functor $\Exit(E) \to \cat{C}$. By pullback along the base
  space the spaces $\Mesh(B, \cat{C})$ arrange into a functor $\Mesh(-, \cat{C}) : \Strat^\op \to \sSet$.
\end{definition}

\begin{lem}\label{lem:mesh-forget-label-fibration}
  Let $B$ be a stratified space and $\cat{C}$ an $\infty$-category. Then the
  canonical map \[\Mesh(B, \cat{C}) \to \BMesh(B) \] is a Kan fibration. In
  particular $\Mesh(-, \cat{C})$ defines a functor
  $\Strat^\op \to \Space$.
\end{lem}
\begin{proof}
  We need to solve lifting problems of the form
  \[
    \begin{tikzcd}
      \Lambda_i \ord{k} \ar[d] \ar[r] & \Mesh(B, \cat{C}) \ar[d] \\
      \Delta\ord{k} \ar[r] \ar[ur, dashed] & \BMesh(B)
    \end{tikzcd}
  \]
  Equivalently we are given a pullback square of mesh bundles
  \[
    \begin{tikzcd}
      E' \ar[r] \ar[d] \pullback & E \ar[d] \\
      B \times \HornTop{k}{i} \ar[r] & B \times \DeltaTop{k}
    \end{tikzcd}
  \]
  together with a functor $f : \Exit(E') \to \cat{C}$ which we need to extend to a functor
  $\Exit(E) \to \cat{C}$. This is possible via the retraction of $E' \to E$ induced by
  a retraction of
  $\HornTop{k}{i} \to \DeltaTop{k}$.
\end{proof}

It now follows that for every mesh bundle $p : E \to B$ the induced functor
$p_* : \Exit(E) \to \Exit(B)$ is a truss bundle via the composition of pullback squares
\[
  \begin{tikzcd}
    \Exit(E) \ar[d, "p_*"] \ar[r] \pullback &
    \EMesh \ar[r] \ar[d] \pullback &
    \ETruss \ar[d] \\
    \Exit(B) \ar[r] &
    \BMesh \ar[r] &
    \FinOrd
  \end{tikzcd}
\]
This induces a natural functor $\Mesh(B, \cat{C}) \to \Truss(\Exit(B), \cat{C})$.

\begin{prop}
  Let $B$ be a stratified space and $\cat{C}$ an $\infty$-category. Then
  \[
    \Mesh(B, \cat{C}) \to \Truss(\Exit(B), \cat{C})
  \]
  is an equivalence of spaces. In particular $\Mesh(-, \cat{C})$ represents
  an $\infty$-category $\Mesh(\cat{C})$ of meshes equivalent to the $\infty$-category
  $\Truss(\cat{C})$ of trusses labelled in $\cat{C}$.
\end{prop}
\begin{proof}
  Consider the commutative square of spaces
  \[
    \begin{tikzcd}
      \Mesh(B, \cat{C}) \ar[r] \ar[d] & \Truss(\Exit(B), \cat{C}) \ar[d] \\
      \BMesh(B) \ar[r] & \CatInfty(\Exit(B), \Delta)
    \end{tikzcd}
  \]
  The lower horizontal map is an equivalence by Proposition~\ref{prop:mesh-classify}
  and so it suffices to check that the upper horizontal map induces an equivalence on homotopy fibres.
  Let $\pi : E \to B$ be a mesh bundle over $B$.
  By Lemma~\ref{lem:mesh-forget-label-fibration} the map $\Mesh(B, \cat{C}) \to \BMesh(B)$ of Kan complexes
  is a Kan fibration and so the homotopy fibre over $\pi$ is given by the ordinary fibre of simplicial sets.
  But then the map $\Mesh(B, \cat{C}) \to \Truss(\Exit(B), \cat{C})$ restricts to the identity on
  $\CatInfty(\Exit(E), \cat{C})$ in the fibres.
\end{proof}

We can now iterate this construction as we did for trusses to obtain $\infty$-categories $\Mesh^n(\cat{C})$.

\begin{definition}
  Let $n \geq 0$, let $B$ be a stratified space and let $\cat{C}$ be an $\infty$-category. Then
  $\Mesh^n(B, \cat{C})$ is the simplicial set whose $k$-simplices are sequences
  of mesh bundles
  \[
    \begin{tikzcd}
      E_n \ar[r, "\pi_n"] &
      E_{n - 1} \ar[r, "\pi_{n - 1}"] &
      \cdots \ar[r] &
      E_1 \ar[r, "\pi_1"] &
      E_0 = \cat{B} \times \DeltaTop{k}
    \end{tikzcd}
  \]
  together with a functor $\Exit(E_n) \to \cat{C}$.
  This
  assembles into a functor $\Mesh^n(-, \cat{C}) : \Strat^\op \to \sSet$
  by pullback of bundles.
\end{definition}

Then by a packing construction similar to Construction~\ref{constr:truss-packing} we can see that
$\Mesh^n(-, \cat{C})$ represents the $\infty$-category $\Mesh^n(\cat{C})$ and so $\Mesh^n(\cat{C})$ classifies $n$-fold mesh bundles with labels in an $\infty$-category $\cat{C}$.
From here the theory can be developed further with notions of normalisation and coarsest refining meshes as in~\cite{fct, Heidemann_2022}.
We leave this for future work.

\bibliographystyle{alpha}
\bibliography{paper}

\end{document}